\documentclass[12pt]{article}

\usepackage{amsmath,amssymb,amsthm}
\usepackage{xparse}
\usepackage[pdftex,bookmarksopen=true,bookmarks=true,unicode,setpagesize]{hyperref}
\usepackage[dvipsnames]{xcolor}
\usepackage{dsfont}
\usepackage{bm}
\usepackage{authblk}
\usepackage{enumitem}

\usepackage{mathtools}

\makeatletter
\DeclareRobustCommand\widecheck[1]{{\mathpalette\@widecheck{#1}}}
\def\@widecheck#1#2{%
    \setbox\z@\hbox{\m@th$#1#2$}%
    \setbox\tw@\hbox{\m@th$#1%
       \widehat{%
          \vrule\@width\z@\@height\ht\z@
          \vrule\@height\z@\@width\wd\z@}$}%
    \dp\tw@-\ht\z@
    \@tempdima\ht\z@ \advance\@tempdima2\ht\tw@ \divide\@tempdima\thr@@
    \setbox\tw@\hbox{%
       \raise\@tempdima\hbox{\scalebox{1}[-1]{\lower\@tempdima\box
\tw@}}}%
    {\ooalign{\box\tw@ \cr \box\z@}}}
\makeatother

\usepackage{scalerel,stackengine}
\stackMath
\newcommand\reallywidehat[1]{%
\savestack{\tmpbox}{\stretchto{%
  \scaleto{%
    \scalerel*[\widthof{\ensuremath{#1}}]{\kern-.6pt\bigwedge\kern-.6pt}%
    {\rule[-\textheight/2]{1ex}{\textheight}}%WIDTH-LIMITED BIG WEDGE
  }{\textheight}%
}{0.5ex}}%
\stackon[1pt]{#1}{\tmpbox}%
}
\stackMath
\newcommand\reallywidecheck[1]{%
\savestack{\tmpbox}{\stretchto{%
  \scaleto{%
    \scalerel*[\widthof{\ensuremath{#1}}]{\kern-.6pt\bigvee\kern-.6pt}%
    {\rule[-\textheight/2]{1ex}{\textheight}}%WIDTH-LIMITED BIG WEDGE
  }{\textheight}%
}{0.5ex}}%
\stackon[1pt]{#1}{\tmpbox}%
}
\numberwithin{equation}{section}
\numberwithin{figure}{section}
\newcounter{dummy} \numberwithin{dummy}{section}
\newtheorem{thm}[dummy]{Theorem}
\newtheorem{lem}[dummy]{Lemma}
\newtheorem{prop}[dummy]{Proposition}
\newtheorem{cor}[dummy]{Corollary}
\newtheorem{defn}[dummy]{Definition}

\newtheorem{exmp}[dummy]{Example}
\newtheorem{rem}[dummy]{Remark}

\newcounter{assum}
\newenvironment{assum}{\refstepcounter{assum}\equation}{\tag{A\theassum}\endequation}
\newtheorem{assumm}[assum]{Assumption}

\newcounter{cque}

\binoppenalty=\maxdimen
\relpenalty=\maxdimen
\date{\today}

\DeclareMathOperator*{\esssup}{ess\,sup}

\renewcommand{\k}{\kappa}

\newcommand{\R}{\mathbb{R}}
\newcommand{\Rc}{\widehat{\mathbb{R}}}
\newcommand{\Rcs}{\Rc_{sym}}
\newcommand{\Rs}{\mathbb{R}_{sym}}
\newcommand{\bR}{\mathbf{R}}
\newcommand{\bRc}{\widehat{\mathbf{R}}}

\newcommand{\bK}{\mathbf{K}}

\newcommand{\B}{\text{B}}
\newcommand{\iD}{\mathit{\Delta}}
\newcommand{\x}{\bm{x}}

\newcommand{\dx}{\tilde{x}}
\newcommand{\z}{\mathbf{z}}
\newcommand{\dt}{\tilde{t}}

\newcommand{\y}{\bm{y}}

\newcommand{\X}{\mathbf{X}}
\newcommand{\T}{\mathbf{T}}
\newcommand{\N}{\mathbb{N}}

\newcommand{\Z}{\mathbb{Z}}

\renewcommand{\L}{\mathfrak{L}}

\newcommand{\la}{\lambda}

\newcommand{\el}{e_\lambda}

\newcommand{\1}{\mathds{1}}

\newcommand{\wh}[1]{\widehat{#1}}
\newcommand{\wc}[1]{\widecheck{#1\,}}
\newcommand{\wwh}[1]{\reallywidehat{#1}}
\newcommand{\wwc}[1]{\reallywidecheck{#1}}
\newcommand{\vrt}[2]{\left. #1 \right\vert_{#2}}

\newcommand\restr[2]{{% we make the whole thing an ordinary symbol
  \left.\kern-\nulldelimiterspace % automatically resize the bar with \right
  #1 % the function
  \vphantom{\big|} % pretend it's a little taller at normal size
  \right|_{#2} % this is the delimiter
  }}
\let\copyint\int
\RenewDocumentCommand \int {o o}
{ \IfNoValueTF {#2} { \IfNoValueTF {#1} { \copyint } { \copyint\limits_{#1} } }	{ \copyint\limits_{#1}^{#2} } }

\NewDocumentCommand \E {o o}
{ \IfNoValueTF {#2} { 
											\IfNoValueTF {#1} {\mathbf{E}} 
																			 	{\mathbf{E}\big[#1\big]} 
										}
										{ 
											\mathbf{E}_{#1}\big[#2\big] 
										}
}

\RenewDocumentCommand \P {o o}
{ \IfNoValueTF {#2} { 
											\IfNoValueTF{#1} {\mathbf{P}} 
																			 {\mathbf{P}\big[#1\big]} 
										}
										{ 
											\mathbf{P}_{#1}\big[#2\big] 
										}
}

\NewDocumentCommand \df {m m o}
{ \frac{\partial^{#3} #1}{\partial #2^{#3}} }

\title{On stability of traveling wave solutions for integro--differential equations related to branching Markov processes.}
\author{Pasha Tkachov \thanks{Email: \texttt{pasha.tkachov@gssi.it}}}
\affil{Gran Sasso Science Institute, L'Aquila}
\date{\vspace{-1.5cm}}

\begin{document}
\maketitle

\begin{abstract}
The aim of this paper is to prove stability of traveling waves for integro-differential equations connected with branching Markov processes.
In other words, the limiting law of the left-most particle of a (time-continuous) branching Markov process with a L\'{e}vy non-branching part is demonstrated.
The key idea is to approximate the branching Markov process by a branching random walk and apply the result of A\"{i}d\'{e}kon \cite{Aid2013} on the limiting law of the latter one.
%Results by McKean \cite{McK1975} and Bramson \cite{Bra1983} for the branching Brownian motion are covered.
\end{abstract}

\section{Introduction}
The Fisher-KPP equation  and its analogues have been attracting growing attention over the last decade (see e.g. \cite{BBD2017,BD2015,CD2005,CDM2008,FKT2018i,HNRR2013,NRR2017} and references therein).
Common results in the PDE literature are existence and uniqueness of traveling waves for speed $c\geq c_*$, where $c_*$ is called the minimal speed of propagation (see references in \cite{FKT2018i}).
Studying of the traveling wave with the minimal speed is of primary interest \cite{ES2000}.
Although, its stability is well known for Fisher-KPP \cite{Bra1983,KPP1937,Lau1985,Uch1978},
to the best of our knowledge, it is absent for various analogues of the equation (see examples below).  
Thus, our goal is to prove stability of the traveling wave with the minimal speed to a large class of Fisher-KPP-type equations demonstrating universality of long-time behaviour in this class.

Ikeda, Nagasawa and Watanabe shown in \cite{INW1968i,INW1968ii,INW1969iii} that a class of branching Markov processes may be associated with non-linear evolution equations.
A process $\X$ from this class started from a one-point configuration may be described as follows: the point moves on the real line\footnote{Ikeda, Nagasawa and Watanabe considered a more general state space instead of $\R$.} as a Markov process $X^0$ up to the first branching time $\tau$, at which $X^0$ is terminated and new particles are instantly born and randomly distributed on the real line according to a stochastic kernel $\pi$.
The new-born particles repeat behavior of its parent: they move independently as $X^0$ and are terminated at independent random times distributed as $\tau$ producing new-born particles according to $\pi$. The procedure continues to infinity.
Such process $\X$ is called a $(X^0,\pi)$-branching Markov process (see Definition~\ref{defn:BMP}).
The configuration space of $\X$ is then $\bR=\bigcup\{\Rs^n\,\vert\, n\geq 0\}$, where $\Rs^n$ denotes the $n$-fold symmetric product of the real line $\R$.
Thus, $\X$ consists of $n\in\N$ points at time $t$ if $\X_t\in \Rs^n$, and $\X$ dies out if $\X_t\in \Rs^0:=\{\varnothing\}$, for some $t>0$.

Ikeda, Nagasawa and Watanabe proved (see Theorem~\ref{thm:BMP_and_PDE}), that the semigroup $\T$ of the $(X^0,\pi)$-branching Markov process $\X$ started from a one-point configuration $\{x\}$, $x\in\R$, is a minimal solution in the class of non-negative functions to the so-called S-equation. Namely, for $0\leq f \leq 1$ - Borel,
$t\geq0$, 
\[
	u(x,t) := \E[\{x\}][\prod_{y\in\X_t} f(y)] =  \vrt{(\T_t\wh{f}\,)}{\R} (x), \]
solves
\begin{equation}\label{eq:S}
	u(x,t) = T_t^0f(x) +  (\bK \wh{u})(x,t),
\end{equation}
where, $f$ is bounded Borel, $\wh{f}(\z) = \prod_{j=1}^n f(z_j)$,  $\wh{u}(\z,t) = \prod_{j=1}^n u(z_j,t)$, $\z\in\Rs^n$; for $X^0 = (X^0_t,\mathbb{P}_x^0)$, 
\begin{gather}
	T_t^0f(x) = \mathbb{E}_x^0[f(X^0_t), t<\tau],  \label{eq:semigroup_of_nonbranching_part} \\
	K(x;dt\,dy) = \mathbb{P}_x^0[\tau\in dt,\ X^0_{\tau-}\!\in dy],\label{eq:position_of_nonbranch_at_first_branch}\end{gather}
and for $g:\bR\times\R_+\to\R$, 
\begin{gather}
	(\bK g)(x,t) = \int[0][t] \int[\R] K(x;dsdy) \int[\bR] \pi(y,d\z) g(\z,t-s). \label{eq:K}
\end{gather}

Our goal is to prove that the minimal non-negative solution to the S-equation converges to a limiting profile.
This is equivalent to existence of the limiting law of the left-most particle of the corresponding branching Markov process $\X$.
We will apply the result by A\"{i}d\'{e}kon \cite{Aid2013} (c.f.  Theorem~\ref{thm:lim_of_leftmost_of_BRW}) on the limiting law of a branching random walk.
The idea is to consider the sampling $\{\X_{\frac{n}{2^k}}\}_{n\geq 0}$, $k\in\N$.
Since $\{\X_{\frac{n}{2^k}}\}_{n\geq 0}$ is a branching random walk, then the result of A\"{i}d\'{e}kon implies convergence of the left-most particle of $\X$ to a limiting profile  on subsets $\{t\in\frac{n}{2^k}\}_{n\in\N_0}$, $k\in\N$.
We will show that the limiting law is independent of $k$.
Finally, taking $k\to\infty$ we will obtain the statement for continuous time.

Together with the S-equation let us consider the following auxiliary linear equations, 
	\begin{align}
		v_\la(x,t) &= (T_t^0 \el)(x) + (\bK \wc{v\!}_\la)(x,t), \label{eq:for_v}\\ 
		w_{\la,\mu}(x,t) &= (T_t^0 e_{\la+\mu})(x) + \bK (\wc{w\!}_{\la,\mu} +\wc{v\!}_\la\wc{v\!}_\mu- \wwc{v_\la v_\mu}) (x,t), \label{eq:for_w}
	\end{align} 
	where $\la,\mu\in\R$, $\el(x):=e^{-\la x}$, $\wc{v}(\x,t) = \sum_{j=1}^n v(x_j,t)$, $\x\in\Rs^n$, $n\geq 0$.
	In Theorem~\ref{thm:f_fg} below we will show that the following functions are the minimal non-negative solutions  to \eqref{eq:for_v} and \eqref{eq:for_w} correspondingly,
		\begin{align}
			v_\la(x,t) &:= \E[\{x\}][\sum_{y\in\X_t} e^{-\la y}] = (\T_t \wc{e\!}_{\la}) (\{x\}), \label{eq:Laplace_tranform}\\	
			w_{\la,\mu}(x,t) &:= \E[\{x\}][\sum_{y\in\X_t} e^{-\la y}\sum_{y\in\X_t} e^{-\mu y}] = (\T_t \wc{e\!}_{\la}\wc{e\!}_{\mu}) (\{x\}). \label{eq:w_la_mu}
		\end{align}

		To formulate the main result of the article we need to impose additional assumptions on $\X$. 
First we exclude the possibility that the process $\X$ may explode in finite time, which is equivalent to the following assumption (see Lemma~\ref{lem:isnt_extinct_is_necessary}):
\begin{assumm}\label{assumm:doesnt_explode}
	For a $(X^0,\pi)$-branching Markov process $\X$, we assume that $u\equiv 1$ is the unique solution of the corresponding S-equation \eqref{eq:S} with the initial value $f\equiv1$.
\end{assumm}
Together with the process $\X$ we consider its sampling $(\X_n)$, $n\in\N_0$.
In view of \eqref{eq:Laplace_tranform}, the log-Laplace transform of a point process $\X_1$ is defined as follows 
\begin{equation}\label{eq:log-Laplace}
	\psi(\la) := \ln v_{\la}(0,1), \quad \la\in\R.
\end{equation}
With the following assumption we ensure that $(\X_n)$ survives with a positive probability and its left-most particle propagates linearly (asymptotically equivalent to $- n c_*$) as $n\to\infty$ on the set of non-extinction,
\begin{assum}\label{assum:supercritical}
	\psi(0)\in(0,\infty)\text{ and }\psi(\la)<\infty\text{ for some } \la>0.
\end{assum}
The speed of propagation $c_*$ is then defined as follows: under \eqref{assum:supercritical}, denote 
\begin{equation}\label{eq:la_0}
	\la_0:=\sup\{s>0: \psi(s)<\infty\}\in(0,\infty],
\end{equation}
and assume that there exists $\la_*\in(0,\la_0]$ satisfying
\begin{assum}\label{assum:c_*}
	\frac{\psi(\la_*)}{\la_*} = \inf_{\la>0} \frac{\psi(\la)}{\la} =: c_*.
\end{assum}
Let us also make the following technical restriction in order to ensure assumptions \cite[(1.1)--(1.4)]{Aid2013}, 
\begin{assum}\label{assum:sufficient_for_H1_and_H2}
	\begin{gathered}
		\la_* < \la_0,\qquad c_* = \frac{\partial}{\partial \la} \psi(\la_*),\\
		\exists\,\delta\in(0,\la_0 - \la_*):\ w_{0,0}(0,1)+w_{0,\la_*}(0,1) + w_{\delta,\la_*}(0,1)<\infty. 
	\end{gathered}
\end{assum}
We say that the distribution of $\X_1$ is \textit{non-lattice} if there do not exist $a>0$, $b\in\R$, such that $\P[\{0\}][\X_1\subset a\Z+b]=1$.
Now we can formulate the main result of the paper.
\begin{thm}\label{thm:main}
	Let $\X$ be a spatially homogeneous $(X_0,\pi)$-branching Markov process, which satisfies Assumption~\ref{assumm:doesnt_explode}, \eqref{assum:supercritical}, \eqref{assum:c_*}, \eqref{assum:sufficient_for_H1_and_H2}.
Suppose that the distribution of $\X_1$ is non-lattice. Then the following statements hold true:
\begin{enumerate}
	\item The left-most particle of $\X$, $M_t:= \min\{y\in\R: y\in\X_t\}$,
	satisfies, 
	\begin{equation}\label{eq:left-most_limiting_law}
  	\lim\limits_{t\to\infty} \P[\{0\}][M_t + c_*t - \frac{3}{2\la_*} \ln t + C \geq -x] = \phi(x), \quad x\in\R,
	\end{equation}
	where
	\begin{equation}\label{eq:traveling_wave}
		\phi(x) = \E[\{0\}][e^{-e^{-\la_* x}D_\infty}],
	\end{equation}
	and $D_\infty$ is the almost sure limit of the derivative martingale of $(\X_n)$ given by \eqref{eq:derivative_martingale}.
\item The minimal non-negative solution $u(x,t)$ to the S-equation \eqref{eq:S} with the initial condition  $u(x,0)=f(x) = \1_{\R_+}(x)$ satisfies,
	\begin{equation}\label{eq:u_and_M}
		u(x,t) = \P[\{x\}][M_t\geq 0] = \P[\{0\}][M_t \geq -x]. 
	\end{equation}
	\item For $u$ and $\phi$ as above the following asymptotic holds
	\begin{equation}\label{eq:stability}
		\lim\limits_{t\to\infty} u(x+c_*t - \frac{3}{2\la_*} \ln t + C, t) = \phi(x), \quad x\in\R.
	\end{equation}
	Moreover, $\phi(x{-}c_*t)$ is a monotone traveling wave solution to the S-equation, namely, $\phi$ is monotone, $(x,t)\to\phi(x{-}c_*t)$ solves \eqref{eq:S}, and 
	\[
		\lim_{x\to+\infty} \phi(x)=1, \quad \lim_{x\to-\infty} \phi(x)=\E[\{0\}][D_\infty=0]<1.
	\]
\end{enumerate}
\end{thm}
\begin{cor}\label{cor:general_initial_condition}
If one consider a more general initial condition $u(x,0)=g(x)$, such that, for some $h>0$,
\[
	\1_{\R_+}(x) \leq g(x) \leq \1_{\R_+}(x+h), \quad x\in\R,
\]
then the comparison principle (see Proposition~\ref{prop:comparison}) immediately implies, that the corresponding solution $u_g(x,t) = \vrt{(\T_t\wh{g})}{\Rc}(x)$ to \eqref{eq:S} satisfies,
\[
	\phi(x) \leq \liminf_{t\to\infty} u_g(x+\theta(t),t) \leq \limsup_{t\to\infty} u_g(x+\theta(t),t) \leq \phi(x+h),
\]
with $\phi$ given by \eqref{eq:traveling_wave} and $\theta(t) = c_*t-\frac{3}{2\la_*}\ln t +C$.
\end{cor}

\section{Discussion}

The relation \eqref{eq:u_and_M} was first shown by McKean \cite{McK1975} for the branching Brownian motion and the Fisher--KPP equation (see Example~\ref{exmp:X3+P2}).
For the branching Brownian motion the limits \eqref{eq:left-most_limiting_law} and \eqref{eq:stability} are known since Uchiyama \cite{Uch1978}, Bramson \cite{Bra1983} and Lau \cite{Lau1985}.

Theorem~\ref{thm:main} shows, that the limiting behaviour \eqref{eq:stability} holds true for a large class of equations, where general reaction terms are possible due to the large set of branching laws $\pi$ of the underlying branching Markov process, and more general than the Brownian motion propagation mechanisms may be considered due to the L\'{e}vy nature of the non-branching part $X^0$.
Namely, by \cite[Theorem~10.5]{Sat1999} and \eqref{eq:non-branching_part} below, the process $X^0$ is a non-branching part of a spatially homogeneous $(X^0,\pi)$-branching Markov process if and only if there exists a L\'{e}vy process $X$ on $\R$, such that $X^0_t = X_t$, $t<\tau$ (here $\tau$ is the first branching time).

Chen \cite{Che2015} proved that in the case of a branching random walk assumptions \eqref{assum:H1} and \eqref{assum:H2} below are necessary and sufficient for the limiting profile \eqref{eq:traveling_wave} to be non-constant.
In our approach it is difficult to check this assumption uniformly with respect to the sampling parameter $k\in\N$ (i.e., for all $\X_{2^{-k}}$, $k\in\N$).
Therefore, \eqref{assum:sufficient_for_H1_and_H2} and $w_{\la,\mu}$ were introduced.
As one can see in the proof of Proposition~\ref{prop:nl+logistic} given in the appendix, it is easy to check \eqref{assum:sufficient_for_H1_and_H2} for particular examples.

The assumption $\la_*<\la_0$ in \eqref{assum:sufficient_for_H1_and_H2} excludes the edge case $\la_*=\la_0$ (note that $\la_*\leq\la_0$ by definition).
In the case of Example~\ref{exmp:X2+P3}, it was shown in \cite{FKT2018ii} that the traveling wave \eqref{eq:traveling_wave} with the minimal speed can have different asymptotic behaviour depending on the choice of $X^0$.
Namely, additionally to the expected asymptotic behaviour $\phi(x)\sim x e^{-\la_*x}$, it is also possible to have $\phi(x) \sim e^{-\la_*x}$. 
Ebert and Saarloos \cite{ES2000} heuristically shown that different asymptotic behaviour of $\phi$ may lead in \eqref{eq:stability} to a correction term different of $\frac{3}{2\la_*}\ln t$.
Therefore, our hypothesis is that if $\la_*=\la_0$, then correction terms different from the one in \eqref{eq:stability} are possible. 

In the case of the branching Brownian motion the representation of the traveling wave \eqref{eq:traveling_wave} is well known, see Lalley and Sellke \cite{LS1987}. Also see \cite{ABBS2013} and references therein for more delicate results.
For the idea of how to cover more general initial conditions to \eqref{eq:S} we refer to \cite{BD2011}.
For a shorter proof of \cite[Theorem~1.1]{Aid2013} see \cite{BDZ2016}.
For resent analytic results on stability of traveling waves in the Fisher-KPP equation we refer to \cite{HNRR2013,NRR2017} and references therein.
How different may be correction terms in \eqref{eq:stability} for Fisher-KPP-type equations see in \cite{BBD2017}.

\section{Related PDEs}
In order to pass from the S-equation \eqref{eq:S} to a partial differential equation (PDE) one need to differentiate both sides of \eqref{eq:S} with respect to the time variable.
This is equivalent to the definition of a generator of the underlying $(X^0,\pi)$-branching Markov process $\X$, that requires additional regularity assumptions on the process $\X$.
We prefer to avoid introduction of such assumptions here as this would increase technicality of the article.
Nevertheless, we will show now heuristically, which PDEs correspond to \eqref{eq:S}, \eqref{eq:for_v} and \eqref{eq:for_w}, that is straightforward to prove rigorously for examples given below.
For the general rigorous approach we refer the reader to the notion of H-regularity in \cite{INW1969iii}.

If $u$ satisfies the S-equation \eqref{eq:S} we will call it a mild solution. Let us call it a strong solution if it satisfies the following PDE, 
\begin{equation}\label{eq:S_equation_PDE}
	\frac{\partial u}{\partial t}(x,t) = (A^0u)(x,t) + k\int[\bR] \pi(x,d\z) \wh{u}(\z,t),
\end{equation}
where $A^0$ is the generator of $X^0$, $k:=\frac{d \mathbb{P}^0_x(\tau\in dt)}{dt}(0)$ - value of the probability density of the branching time $\tau$ at $0$, $u(x,0)=f(x)$, $x\in\R$.
The Laplace transform $v_\la$ of $\X$ defined by \eqref{eq:Laplace_tranform} satisfies \eqref{eq:for_v}, hence,
\begin{equation}\label{eq:laplace_transform_PDE}
	\frac{\partial v_\la}{\partial t}(x,t) = (A^0v_\la)(x,t) + k\int[\bR] \pi(x,d\z) \wc{v\!}_\la(\z,t),
\end{equation}
with $v_\la(x,0)=e^{-\la x}$, $x\in\R$.
Similarly, the function $w_{\la,\mu}$ defined by \eqref{eq:w_la_mu} satisfies \eqref{eq:for_w}, therefore, 
\begin{align}
	\frac{\partial w_{\la,\mu}}{\partial t}(x,t) &= (A^0w_{\la,\mu})(x,t) \nonumber \\
		&\quad + k\int[\bR] \pi(x,d\z) (\wc{w\!}_{\la,\mu}(\z,t)+ (\wc{v\!}_\la \wc{v\!}_\mu - \wc{v_\la v_\mu})(\z,t)). \label{eq:w_la_mu_PDE}
\end{align}

\section{Examples}
\textbf{I. Non-branching parts.}
\begin{enumerate}[label=(\textnormal{X\arabic*}), wide]
	\item \label{item:X1} 
		\textit{Constant.} Let the non-branching part $X^0$ be trivial: the point does not move and dies with a random exponentially distributed time with rate $1$.
Thus, for $f\in\B(\R)$, the generator $T^0$ of the process $X^0$ has the following form,	
	\begin{align*}
		T_t^0f(x) &= \mathbb{E}_x^0[f(X^0_t), \tau>t] = e^{-t} f(x).
	\end{align*}
\item \label{item:X2} 
\textit{Pure-jump process.} Let the non-branching part $X^0$ of a branching Markov process $\X$ be the pure-jump Markov process with a bounded jump-kernel $a\in L^1(\R{\to}\R_+)$ and the jump rate $1$.
Namely, starting from a point $x\in\R$, the process $X^0$ waits a random exponentially distributed time with rate $1$, and, then, it jumps from $x$ to a point $y\in\R$ with probability $a(y-x)dy$.
Next, we suppose, that the branching time $\tau$ is exponentially distributed with rate $1$.
At time $\tau$ the particle $X^0_{\tau-}$ dies. 
Thus, for $f\in\B(\R)$, the generator $T^0$ of the process $X^0$ has the following form,	
	\begin{align*}
		T_t^0f(x) &= \mathbb{E}_x^0[f(X^0_t), \tau>t] = e^{-t} \int[\R]p(y-x,t)f(y)dy,
	\end{align*}
	where 
	\begin{gather}
		p(y-x,t)dy = e^{-t}\big( \delta_x(dy) + \sum\limits_{n\in\N} \frac{t^n}{n!} a^{*n}(y-x)dy\big), \label{eq:jump_transition_density} \\ 
		a^{*n}(y)  = (a*a^{*(n-1)})(y), \quad a^{*2}(y) = \int[\R] a(y-z)a(z)dz, \quad a^{*1}(y)=a(y). \nonumber 
	\end{gather}

\item \label{item:X3} \textit{Standards Brownian motion.} Let the non-branching part $X^0$ be a standard Brownian motion up to the branching time $\tau$,
	which is exponentially distributed with rate $1$.
	Then, for $f\in\B(\R)$, the generator $T^0$ of the process $X^0$ has the following form,	
	\begin{align*}
		T_t^0f(x) &= \mathbb{E}_x^0[f(X^0_t), \tau>t] = e^{-t} \int[\R]p(y-x,t)f(y)dy,
	\end{align*}
	where 
	\[
		p(x,t) = \frac{1}{\sqrt{2\pi t}} e^{-\frac{x^2}{2t}}.
	\]
\end{enumerate}
\textbf{II. Branching laws.} 
\begin{enumerate}[label=(\textnormal{P\arabic*}), wide]
	\item \label{item:P1}
	Let a particle at the moment of its death gives birth to two children which are positioned at the same point, where the parent dies. Then the branching law $\pi$ has the following form,
	\begin{align*}
		\pi(y,d\z) &= \1_{\Rs^2}(\z) \delta_{y}(dz_1)\delta_{y}(dz_2).
	\end{align*}

	\item \label{item:P2}
		We consider the following generalization of the branching law \ref{item:P1}. We assume that a particle gives birth to $n$ children with a probability $p_n$, and children are placed at the same point where the parent dies.
	\[
		\pi(y,d\z) = \sum\limits_{n\in\N} p_n \1_{\Rs^n}(y) \prod_{j=1}^n \delta_y(d z_j), \quad p_n\in[0,1], \quad \sum_{n\in\N} p_n \leq 1.
	\]

	\item \label{item:P3}
	Let a particle at the moment of its death gives birth to two children, one of which is positioned at the same point where the parent dies, and the second one is placed randomly at $z\in\R$ with a probability $b(z-X^0_{\tau-})dz$, where $b$ is a bounded probability density. Hence, the branching law is defined as follows,
	\begin{align*}
		\pi(y,d\z) = \1_{\Rs^2}(\z) \delta_{y}(dz_1) b(z_2-y)dz_2.
	\end{align*}
\end{enumerate}

\textbf{III. $(X^0,\pi)$-branching Markov processes}

\begin{exmp}[X1+P2]\label{exmp:X1+P2}
	Let the non-branching part $X^0$ of a $(X^0,\pi)$-branching Markov process $\X$ be defined by \ref{item:X1} and its branching law $\pi$ by \ref{item:P2}.
Then $\X$ is the Galton-Watson process. The corresponding S-equation is then,
\[
	\partial_t u(t) = -u(t) + (Fu)(t),
\]
where we omit the trivial dependence on $x\in\R$.
This example is standard, so we skip farther details and refer to \cite{AN1972}.
\end{exmp}

\begin{exmp}[X1+P3]\label{exmp:X1+P3}
	Let the $(X^0,\pi)$-branching Markov process $\X$ be defined by \ref{item:X1} and \ref{item:P3}. Then the S-equation reads as follows, 
\[
	\partial_t u(x,t) = -u(x,t) + u(x,t) (b*u)(x,t).
\]
Such equation first appeared in \cite{Mol1972,Mol1972a} (apply the change of variables $u\to1-u$, $b(x)\to b(-x)$). Existence of traveling waves was proven in \cite{Mol1972a}. To the best of our knowledge, stability of the traveling wave with the minimal speed given by Theorem~\ref{thm:main} was absent in the literature.
The Laplace transform $v_\la(x,t)$ of $\X$, defined by \eqref{eq:Laplace_tranform} satisfies \eqref{eq:for_v}, which reads now as follows,
  \[
  	\partial_t v_\la(x,t) = (b*v_\la)(x,t),\quad v_\la(x,0)=e^{ -\la x}, \quad x\in\R,\ t>0.
  \]
  Therefore, the log-Laplace transform of $\X_1$ equals   
  \[
  	\psi(\la) = \ln v_\la(0,1) = (\L b)(\la), \qquad (\L b)(\la) := \int[\R] b(x) e^{ - \la x} dx. 
  \]

\begin{prop}
 	Let there exist $l,\delta,\la>0$ such that,
	\begin{equation}
		I:=\inf_{y\in(-l-\delta,-l)} b(y) >0, \qquad (\L b)(\la)<\infty.
	\end{equation}
	Suppose $\la_0$ be defined by \eqref{eq:la_0} and infimum in \eqref{assum:c_*} is not attained at $\la_*=\la_0$.
	Then the $(X_0,\tau)$-branching Markov process $\X$ given by \ref{item:X1} and \ref{item:P3} satisfies conditions of Theorem~\ref{thm:main}.
\end{prop}
We omit the proof of the proposition, since it is similar to the one of Proposition~\ref{prop:nl+logistic} below, which we prove in the appendix. 
\end{exmp}

\begin{exmp}[X2+P1]\label{exmp:X2+P1}
	Let the $(X^0,\pi)$-branching Markov process $\X$ be defined by \ref{item:X2} and \ref{item:P1}. 
	Then, the corresponding S-equation \eqref{eq:S} has the following form, 
  \[
  	\partial_t u(x,t)	 = (a*u)(x,t) - 2u(x,t) + u^2(x,t).
  \]
	Such equation is a special case of the equation considered, for example, in \cite{CD2005,CDM2008}, where existence of traveling waves and formula for the speed \eqref{assum:c_*} were proven.
	To the best of our knowledge, stability of the traveling wave with the minimal speed, which is exactly the one given in Theorem~\ref{thm:main}, was not shown before. 
The idea to consider jumps as a mechanism of propagation was suggested already in the seminal paper by Kolmogorov, Petrovskii and Piskunov \cite{KPP1937}, where the authors approximated $a*u-u$ by the Laplace operator.
The idea to consider the logistic reaction term $u-u^2$ is usually referred to Fisher \cite{Fis1937}.
 
  The Laplace transform $v_\la(x,t)$ of $\X$, defined by \eqref{eq:Laplace_tranform} satisfies \eqref{eq:for_v}, which reads now as follows,
  \[
  	\partial_t v_\la(x,t) = (a*v_\la)(x,t),\quad v_\la(x,0)=e^{ -\la x}, \quad x\in\R,\ t>0.
  \]
  Therefore, the log-Laplace transform of $\X_1$ equals   
  \[
  	\psi(\la) = \ln v_\la(0,1) = (\L a)(\la), \qquad (\L a)(\la) := \int[\R] a(x) e^{ - \la x} dx. 
  \]
	The function $w_{\la,\mu}$ given by \eqref{eq:w_la_mu} satisfies \eqref{eq:for_w}. Hence, $w_{\la,\mu}(x,0) = e^{-(\la+\mu)x}$, and 
	\[
		\partial_t w_{\la,\mu}(x,t) = (a*w_{\la,\mu})(x,t) + 2v_\la(x,t)v_\mu(x,t).
	\]

\begin{prop}\label{prop:nl+logistic}
 	Let there exist $l,\delta,\la>0$ such that,
	\begin{equation}\label{eq:safficient_for_pure_jump_logistic}
		I:=\inf_{y\in(-l-\delta,-l)} a(y) >0, \qquad (\L a)(\la)<\infty.
	\end{equation}
	Suppose $\la_0$ be defined by \eqref{eq:la_0} and infimum in \eqref{assum:c_*} is not attained at $\la_*=\la_0$.
	Then the $(X_0,\tau)$-branching Markov process $\X$ given by \ref{item:X2} and \ref{item:P1} satisfies conditions of Theorem~\ref{thm:main}.
\end{prop}
Proposition~\ref{prop:nl+logistic} is proved in the appendix.
\end{exmp}

\begin{exmp}[X2+P2]
	For the $(X_0,\pi)$-branching Markov process $\X$ defined by \ref{item:X2} and \ref{item:P2}, the S-equation \eqref{eq:S} reads as follows,
  \[
		\partial_t u(x,t)	 = (a*u)(x,t) - 2 u(x,t) + F(u(x,t)), \quad F(x) = \sum\limits_{j\in\N} p_j x^j.
  \]
	If $p_0=0$, then $F(0)=0$.
If $p_1<1$, then $F'(1) = \sum_{n\geq1} np_n>1$.
As a result, after the change $u\to1-u$ we arrive to the Fisher-KPP-type equation with the KPP reaction term.
	We refer again to \cite{CD2005,CDM2008}.

	The Laplace transform $v_\la(x,t)$ of $\X$ satisfies,
  \[
		\partial_t v_\la(x,t) = (a*v_\la)(x,t) - 2v_\la(x,t) + (\sum_{n\geq1} np_n) v_\la(x,t),\quad v_\la(x,0)=e^{ -\la x}.
  \]
  Therefore, the log-Laplace transform of $\X_1$ equals   
  \[
		\psi(\la) = \ln v_\la(0,1) = (\L a)(\la) - 2 + \sum_{n\geq1} np_n.
  \]
	Proposition~\ref{prop:nl+logistic} holds true in this case too.
\end{exmp}

\begin{exmp}[X2+P3]\label{exmp:X2+P3} 
	The $(X_0,\pi)$-branching Markov process $\X$ defined by \ref{item:X2} and \ref{item:P3} corresponds to the following  S-equation \eqref{eq:S},
  \[
  	\partial_t u(x,t)	 = (a*u)(x,t) - 2 u(x,t) + u(x,t)(b*u)(x,t).
  \]
	This equation appears as a scaling limit of a birth-death point process in the so-called Bolker-Pacala model \cite{BP1997} (consider $u\to1-u$, $a(x)\to a(-x)$, $b(x)\to b(-x)$).
	See \cite{FM2004}, \cite{FKK2011} for the rigorous derivation, and \cite{FKT2018i,FKT2018ii} for existence and uniqueness of traveling waves as well as the derivation of the formula \eqref{assum:c_*} of the minimal speed. 

	The Laplace transform $v_\la(x,t)$ of $\X$ satisfies,
  \[
  	\partial_t v_\la(x,t) = ((a+b)*v_\la)(x,t) - v_\la(x,t),\quad v_\la(x,0)=e^{ -\la x}.
  \]
  Therefore, the log-Laplace transform of $\X_1$ equals   
  \[
  	\psi(\la) = \ln v_\la(0,1) = (\L a)(\la) + (\L b)(\la) - 1.
  \]
	\begin{prop}\label{prop:nl+nl_logistic}
 		Let there exist $l,\delta,\la>0$ such that,
		\begin{equation*}
			I:=\inf_{y\in(-l-\delta,-l)} a(y)+b(y) >0, \qquad (\L a)(\la) + (\L b)(\la)<\infty.
		\end{equation*}
		Suppose $\la_0$ be defined by \eqref{eq:la_0} and infimum in \eqref{assum:c_*} is not attained at $\la_*=\la_0$.
		Then the $(X_0,\tau)$-branching Markov process $\X$ given by \ref{item:X2} and \ref{item:P3} satisfies conditions of Theorem~\ref{thm:main}.
	\end{prop}
	
	We omit the proof, since it repeats the one of Proposition~\ref{prop:nl+logistic}.
\end{exmp}

\begin{exmp}[X3+P2]\label{exmp:X3+P2}
	The $(X_0,\pi)$-branching Markov process $\X$ defined by \ref{item:X3} and \ref{item:P2} correspond to the following S-equation
	\[
		\partial_t u(x,t) = \frac{1}{2} \partial^2_{xx} u(x,t) - u(x,t) + F(u(x,t)), \quad F(u) = \sum\limits_{j\in\N} p_j u^j.
	\]
	After the change of variables $u\to1-u$ we arrive to the Fisher-KPP equation \cite{Fis1937,KPP1937}.
	The Laplace transforms $v_\la(x,t)$ satisfies, 
	\[
		\partial_t v_\la(x,t) = \frac{1}{2} \partial^2_{xx} v_\la(x,t) + (\sum_{n\geq1} np_n-1) v_\la(x,t), \quad v_\la(x,t) = e^{\frac{\la^2 t}{2} + (\sum_{n\geq1}np_n -1)t -\la x}.
	\]
 	As a result,
	\[
		\psi(\la) = \frac{\la^2}{2} + (\sum_{n\geq1}np_n -1), \qquad c_*=\sqrt{\sum 2np_n -2}, \qquad \la_*=\sqrt{\sum np_n -1},
	\]
	where the well known formula for the minimal speed of the traveling wave in the Fisher-KPP equation is obtained. Theorem~\ref{thm:main} in this case does not state anything new, see \cite{Bra1983,Lau1985,LS1987,McK1975,Uch1978}.
\end{exmp}

\begin{exmp}[X3+P3]\label{exmp:X3+P3}
	The $(X_0,\pi)$-branching Markov process $\X$ defined by \ref{item:X3} and \ref{item:P3} correspond to the following S-equation 
	\[
		\partial_t u(x,t) = \frac{1}{2} \partial^2_{xx} u(x,t) - u(x,t) + u(x,t)(b*u)(x,t). 
	\]
	The Laplace transform satisfies,
	\[
		\partial_t v_\la(x,t) = \frac{1}{2} \partial^2_{xx} v_\la(x,t) + (b*v)_\la(x,t), \quad v_\la(x,t) = e^{\frac{\la^2 t}{2} + t(\L a)(\la) - \la x}.
	\]
	The log-Laplace transform has the following form,
	\[
		\psi(\la) = \frac{\la^2}{2} + (\L a)(\la).
	\]
\end{exmp}

\section{A relation between branching Markov processes and evolution equations}

The purpose of this section is to formulate results of Ikeda, Nagasawa and Watanabe on connection between branching Markov processes and evolution equations.  We will follow the notations of \cite{INW1968i,INW1968ii,INW1969iii}.

We will denote $\N_0=\N\cup\{0\}=\{0,1,2,\dots\}$,
$\Rc=\R\cup\{\infty\}$ - the one-point (Alexandroff) compactification of the real line $\R$.
We write $\Rs^n$ for $n$-fold symmetric product of $\R$ (i.e. we identify permutations of coordinates), 
and $\Rcs^n$ for $n$-fold symmetric product of $\Rc$. 
Denote by $\bR=\bigcup\{ \Rs^n\,\big\vert\, n\in\N_0\}$ the topological sum of $\Rs^n$, where $\Rs^0=\{\varnothing\}$, $\varnothing$ an extra point, 
\vspace{-3mm}
\[
	\wh{\bR}=\big(\bigcup\{\Rcs^n\,\big\vert\,n\in\N_0\}\big) \cup \{\iD\},
	\vspace{-3mm}
\] 
the topological sum of $\Rcs^n$, where $\Rcs^0=\{\varnothing\}$, $\varnothing$ and $\iD$ - extra points, so that $\bRc\backslash\{\iD\}$ is compactified by $\{\iD\}$.
The set of the bounded Borel real-valued functions on $\R$, $\Rc$, $\bR$, and $\wh{\bR}$ will be denoted by $\B(\R)$, $\B(\Rc)$, $\B(\bR)$ and $\B(\wh{\bR})$ correspondingly.
A function $f$ defined on $\R$ will be always extended to $\Rc$ by $f(\infty)=0$.
The norm of $f\in\B(\Rc)$ and $g\in \B(\wh{\bR})$ will be denoted correspondingly 
\[
	\|f\|= \esssup_{x\in\Rc} |f(x)|, \qquad  \|g\|= \esssup_{x\in\wh{\bR}} |g(x)|.
\]
The bold symbols $\x$, $\y$, $\z$, $\X$, $\T$, $\P$, $\E$ will be used for objects related with the spaces $\bR$, $\bRc$.

For any $f\in\B(\Rc)$, $\|f\|\leq1$, denote   
\begin{equation}\label{eq:hat}
	\wh{f}(\x) = \left\{
		\begin{aligned}
			&1, &\quad& \x=\varnothing,&\\
			&f(x_1)f(x_2)\dots f(x_n), &\quad& \x=\{x_1,x_2, \dots , x_n\}\in\Rs^n,&\\
			&0, &\quad& \x=\iD.&
		\end{aligned}
		\right.
\end{equation}
Obviously, $\wh{f}\in \B(\wh{\bR})$. %If $\|f\|<1$ then $\wh{f}\in\B(\bRc\backslash\{\iD\})$ and vanishes as $\x\to\iD$.

We will use Dynkin's setup for Markov processes and refer the reader to \cite{INW1968i} and \cite{Sha1988} for more details.
We will always assume that the filtered probability space satisfies the usual conditions, meaning that it is complete and the filtration is right-continuous.
 
\begin{defn}\label{defn:BMP}
Let $\X=(\X_t, \P_{\x})$ be a right-continuous temporally homogeneous Markov process on $\wh{\bR}$, and let $\T_t$ be the transition semi-group on $\B(\wh{\bR})$ induced by $\X$, i.e.,
\begin{equation}\label{eq:MP_semigroup}
	\T_tf(\x)=\E[\x][f(\X_t)],\quad \x\in\wh{\bR},\ t\geq0,\ f\in\B(\wh{\bR}).
\end{equation}
Then the Markov process $\X$ on $\wh{\bR}$ is called a branching Markov process if it satisfies
	\begin{equation}\label{eq:branching_property}
		\T_t\wh{f}(\x) =\!\!\wwh{\vrt{\ (\T_t\wh{f}\,) }{\wh{\R}}}(\x), \quad \x\in\wh{\bR},\ t\geq0. 
	\end{equation}
	for every $f\in\B(\Rc)$, $\|f\|<1$.
\end{defn}

Let $\{\tau_j\}_{j\geq1}$ be splitting (or branching) times of the process $\X$. We denote 
\begin{equation}\label{eq:explosion_and_first_branching}
	\tau_* = \lim\limits_{n\to\infty} \tau_n,\qquad \tau=\tau_1.
\end{equation}
We also write $\tau_\iD$ for the hitting time of $\iD$ %and $\tau_\varnothing$ for the hitting time of $\varnothing$, namely, 
\begin{equation}\label{eq:explosion_and_extinction}
	\tau_\iD = \inf\{t: \X_t = \iD\}. %, \qquad  \tau_\varnothing = \inf\{t: \X_t = \varnothing\}.
\end{equation}
The hitting time of $\infty \in\Rc$ by one of the points in $\X$ will be denoted by 
\begin{equation}\label{eq:first_infty}
	\tau_\infty = \inf\{t: \exists\, y\in\X_t,\ y = \infty\}.
\end{equation}
A large class of branching Markov processes may be defined by its behaviour up to the first branching time $\tau$ and distribution of the new-born particles at $\tau$. Let us make this definition rigorous.

For a branching Markov process $\X$ started from $\X_0=\x=\{x\}\in\Rc$ we can define a new Markov process on $\Rc\cup\{\iD\}$ in the following way
\begin{equation}\label{eq:non-branching_part}
	X_t^0 = \X_t \1_{t<\tau} + \iD\1_{t\geq\tau}, \quad t\geq0.
\end{equation}
Thus $X^0=(X_t^0, \mathbb{P}^0_x)$ describes the behavior of a particle of $\X$ before its branching time $\tau$, and $X^0$ is terminated at $t=\tau$.
We call $X^0$ \textit{the non-branching part} of $\X$.
\begin{defn}[{c.f.~\cite[Definition~1.6]{INW1968i}}]\label{defn:X0_pi_BMP}
	Let $\X$ be a branching strong Markov process, which satisfies the following conditions 
	\begin{enumerate}
		\item \label{assum:leftlim_of_BMP_at_branching_time_exists} $\lim\limits_{n\to0} \X_{\tau-\frac{1}{n}}$ exists almost surely on $\{\tau<\infty\}$. 
	  \item \label{assum:branching_law} There exists a stochastic kernel $\pi(x,E)\text{ on }\Rc\times\wh{\bR}$ such that for each $\la>0,\ x\in\Rc$, and $E$ - Borel in $\wh{\bR}$, we have a.s. on $\{\tau<\infty\}$,
		$$\E[\{x\}][e^{-\la\tau},\, \X_{\tau}{\in} E\,\big\vert\, \X_{\tau-}] = \pi (\X_{\tau-},\, E) \E[\{x\}][e^{-\la\tau}\,\big\vert\,\X_{\tau-}].$$ 
		\item \label{assum:C1} $\P[\x][ \tau_* = \tau_\iD,\ \, \tau_*<\infty] = \P[\x][\tau_* <\infty]$. 
		\item \label{assum:C2} $\P[\x][\tau=s]=0, \quad s\geq0.$ 
		\item \label{assum:inftyinfty} $\P[\x][\tau_\infty<\infty]=0,\quad \x\in\Rs^n,\ n\in\N$.
	\end{enumerate}
	Let $X^0$ be the non-branching part of $\X$. Then we shall call $\X$ the $(X^0,\pi)$-branching Markov process.
\end{defn}
\begin{rem}\label{rem:on_def_of_X0_tau_BMP}
	\begin{enumerate}
		\item If a branching Markov process $\X$ satisfies items \ref{assum:leftlim_of_BMP_at_branching_time_exists} and \ref{assum:branching_law} in Definition~\ref{defn:X0_pi_BMP}, then we call $\pi(x,E)$ the branching law of $\X$.
		\item The item \ref{assum:inftyinfty} in Definition~\ref{defn:X0_pi_BMP} means that if non of the points in $\X_0$ equals $\infty$ then the same holds true for all point in $\X_t$, $t\geq0$.
			This assumptions was not required in \cite[Definition~1.6]{INW1968i}.
		\item By \cite[Theorem~4.4]{INW1969iii}, for $T^0_t$ defined by \eqref{eq:semigroup_of_nonbranching_part}, $K$ defined by \eqref{eq:position_of_nonbranch_at_first_branch}, and a stochastic kernel $\pi$ on $\Rc\times\bRc$, there exists a $(X^0,\pi)$-branching Makrov process. By \cite[Corollary,\,p.\,273]{INW1968i}, $(X^0,\pi)$-branching Markov processes with the same non-branching part and branching law are equivalent (have the same finite dimensional distribution). 
	\end{enumerate}
\end{rem}

The following Lemma shows that Assumption~\ref{assumm:doesnt_explode} is equivalent to the fact that the branching Markov process does not explode in finite time (c.f. \cite[p.115,\ Corollary~3]{INW1969iii}). 
\begin{lem}\label{lem:isnt_extinct_is_necessary}
Assumption~\ref{assumm:doesnt_explode} holds if and only if, (c.f. \eqref{eq:explosion_and_extinction})
	\[
		\P[\{x\}][\tau_\iD = \infty] = 1, \qquad x\in\Rc.
	\]
	Assumption~\ref{assumm:doesnt_explode} and item \ref{assum:C1} of Definition~\ref{defn:X0_pi_BMP} imply, (c.f. \eqref{eq:explosion_and_first_branching})
	\[
		\P[\{x\}][\tau_*<\infty]=0, \qquad x\in\Rc.
	\]
\end{lem}
\begin{proof}
	The statement follows from \cite[(1.10)]{INW1968i}. Namely, for $\x\in\bRc\backslash\{\varnothing,\iD\}$, $t\geq0$,
\begin{equation}\label{eq:explosion}
	\P[\x][\tau_\iD > t] = \T_t\wh{f}(\x), \quad f\equiv1.
\end{equation}
\end{proof}

The following theorem states that the semigroup of the $(X^0,\pi)$-branching Markov process $\X$ started from a one-point configuration $\{x\}$, $x\in\R$, is a minimal solution to the so-called S-equation
\begin{thm}\label{thm:BMP_and_PDE}
	Let $\T_t$ be the semi-group of a $(X^0,\pi)$-branching Markov process and Assumption~\ref{assumm:doesnt_explode} hold.
Then, for $f \in \B(\R)$, $0\leq f \leq1$, 
\begin{equation}\label{eq:minimal_sol_to_S_equation}
	u(x,t) = \vrt{(\T_t\wh{f}\,)}{\R} (x) = \E[\{x\}][\wh{f}(\X_t)],\quad x\in\R,\ t\geq0,
\end{equation}
is the minimal solution to the S-equation \eqref{eq:S} with the initial value $f$.
Moreover, for $u_0(x,t)\equiv0$, and $u_n(x,t) = T_t^0f(x) + (\bK \wh{u}_{n-1})(x,t)$,
\[
	u(x,t) = \lim\limits_{n\to\infty} u_n(x,t).
\] 
\end{thm}

\begin{proof}
	We remind, that we extend $f$ to $\Rc$ by $f(\infty)=0$.
	By item \ref{assum:inftyinfty} of Definition~\ref{defn:X0_pi_BMP} and Assumption~\ref{assumm:doesnt_explode}, if $X_0=\{x\}$, $x\in\R$, then $\X_t\in\bR$, $t\geq0$.
 Thus, statement of the theorem follows from \cite[p.114, Corollary~2]{INW1969iii}.
\end{proof}

For $f\in\B(\Rc)$, denote   
\begin{equation}\label{eq:check}
	\wc{f}(\x) = \left\{
		\begin{aligned}
			&0, &\quad& \x\in\{\varnothing, \iD\},&\\
			&f(x_1)+\dots+f(x_n), &\quad& \x=\{x_1, \dots , x_n\}\in\Rcs^n.&
		\end{aligned}
		\right.
\end{equation}
Obviously, $\wc{f}\in \B(\bRc)$.
%We will also write $(f)\!\wc{\phantom{)}} := \wc{f}$.
We will also write $(f)^{\wc{\phantom{j}}}=\wc{f}$. 

\begin{thm}\label{thm:f_fg}
	Let Assumption~\ref{assumm:doesnt_explode} hold, and for $f,g\geq0$ -- Borel on $\R$ (possibly unbounded), 
	\begin{align}
		v(x,t) &= (\T_t\wc{f})(\{x\}), \label{eq:expectation_semigroup}\\ 
		w(x,t) &=\T_t(\wc{f}\wc{g})(\{x\}), \label{eq:ww}
	\end{align}
	are finite for $x\in\R$, $t\in[0,T]$. Then $v$ and $w$ are the minimal solutions in the class of non-negative functions to the following equations correspondingly
	\begin{align}
		v(x,t) &= (T_t^0 f)(x) + (\bK \wc{v})(x,t), \label{eq:integral_eq_of_exp_semigroup}\\
		w(x,t) &= (T_t^0 fg)(x) + \bK (\wc{w} +\wc{v\!}_f\wc{v\!}_g- \wwc{v_f v_g}) (x,t), \label{eq:w}
	\end{align} where $x\in\R$, $t\in[0,T]$, and 
	\[
		\wc{v} = \wc{v(\cdot,t)}(\x),\ \ v_f(x,t)=(\T_t\wc{f})(\{x\}), \ \ v_g(x,t) = (\T_t\wc{g})(\{x\}).
	\]
\end{thm}
\begin{rem}\label{rem:exp_sem}
	The semigroup $v$ given by \eqref{eq:expectation_semigroup} is called the expectation semigroup.	
\end{rem}

To prove Theorem~\ref{thm:f_fg}, we will need the following variation of \cite[Lemma~4.7,\,4.8]{INW1969iii}. 
\begin{lem}\label{lem:prod_of_sum}
	Let $\X$ be $(X_0,\pi)$-branching Markov process and Assumption~\ref{assumm:doesnt_explode} hold. Then for $f\geq0$, $g\geq 0$, $f,g\in\B(\R)$, $\gamma\in(0,1]$, $k,l\in\N_0$, $\x\in\Rs^n$, $x\in\R$, $t\geq0$, 
	\begin{align}
		\T_t\big(\wh{\gamma}\,(\wc{f})^k(\wc{g})^l\big)(\x) = \sum_{\substack{k_1+\dots+k_n=k\\k_1,\dots,k_n\in\N_0}} &\sum_{\substack{l_1+\dots+l_n=l\\l_1,\dots,l_n\in\N_0}} \frac{k!}{k_1!\dots k_n!} \frac{l!}{l_1!\dots l_n!} \times \nonumber\\
		& \times\prod_{j=1}^n \T_{t-s} \big( \wh{\gamma}\,(\wc{f})^{k_j} (\wc{g})^{l_j}\big)(\{x_j\}). \label{eq:branching_for_product} 
	\end{align}
	\begin{align}
		&\T_t \big(\wh{\gamma}(\wc{f})^k(\wc{g})^l\big)(\{x\}) = (T_t^0 \gamma\, f^k g^l)(x) + \big(\bK\, \T_{t-s}\big(\wh{\gamma}\,(\wc{f})^k(\wc{g})^l\big)\big)(x,t) \label{eq:fg},
	\end{align}
	where $\wh{\gamma}(\z) := \gamma^{n},\ \z\in\Rs^n$.
	Moreover, for $\gamma\in(0,1)$, $\T_t\big(\wh{\gamma}\,(\wc{f})^k(\wc{g})^l\big)(\{x\})$ is the minimal solution to \eqref{eq:fg}.
\end{lem}
\begin{proof}
	Let $\gamma\in(0,1)$. Then there exists $\la_0>0$ such that
  \[
  	\gamma \|e^{\la f} e^{\mu g}\| < 1, \quad |\la|\leq \la_0,\ |\mu|\leq \la_0.
  \]
	We have $\wwh{e^{\la f+\mu g}} = e^{\la \wc{f}+ \mu \wc{g}}$, and for $\y\in\bR\backslash\{\varnothing\},\ |\la|<\la_0,\ |\mu|<\mu_0$,
	\begin{equation}\label{eq:exp_series_expansion}
		\wwh{e^{\la f+\mu g}}(\y) = \sum_{k,l\in\N_0} \frac{\la^k}{k!}\frac{\mu^l}{l!} (\wc{f})^k(\y)(\wc{g})^l(\y).
	\end{equation}
	By the branching property \eqref{eq:branching_property},
	\[
		\T_t(\wwh{\gamma e^{\la f+\mu g}})(\x) = \prod_{i=1}^n \T_t(\wwh{\gamma e^{\la f+\mu g}})(\{x_i\}).
	\]
	Expanding both sides of the equation into series with respect to  $\mu$ and $\la$ for $|\mu|<\la_0$, $|\la|<\la_0$ and collecting terms of the same order we obtain \eqref{eq:branching_for_product} for $\gamma\in(0,1)$. By the monotone convergence theorem, taking $\gamma\to1_{-}$ we extend \eqref{eq:branching_for_product} to $\gamma=1$.

	Next, by Theorem~2.3, 
  \[
  	q(x,t) := \T_t ( \wwh{\gamma e^{\la f+\mu g}})(\{x\}),\quad x\in\R,\ t\geq0,
  \]
	is the minimal solution to the S-equation
	\begin{equation}\label{eq:q}
		q(x,t) = T_t^0 (\gamma e^{\la f+\mu g}) + (\bK\, \wh{q}\,)(x,t). 
	\end{equation}
  Then, by \eqref{eq:exp_series_expansion} and \eqref{eq:q},
	\begin{align}
		q(x,t) &= \sum_{k,l\in\N_0} \frac{\la^k}{k!} \frac{\mu^l}{l!} \T_t(\wh{\gamma}\, (\wc{f})^k(\wc{g})^l)(\{x\}) \nonumber \\ 
		&= \sum_{k,l\in\N_0}  \frac{\la^k}{k!} \frac{\mu^l}{l!} \Big[ T_t^0(\gamma\, f^k g^l)(x) + (\bK\, \T_{t-s}(\wh{\gamma}\,(\wc{f})^k(\wc{g})^l))(x,t) \label{eq:w_series} \Big].
	\end{align}
Comparing the coefficients of $\la^k\mu^l$ and taking $\gamma\to1_-$ we have by the monotone convergence theorem that \eqref{eq:fg} holds for all $\gamma\in(0,1]$.
	Note that $\T_t(\gamma (\,\wc{f}\,)^k(\,\wc{g}\,)^l)(\{x\})$ is the minimal solution to the second equality in \eqref{eq:w_series} in the class of non-negative functions, otherwise we could substitute the minimal one instead of it in \eqref{eq:w_series} and violate the minimality of $q$ to \eqref{eq:q}.
\end{proof}

\begin{cor}\label{cor:nice_formulas}
	For $f,g\geq0$ -- Borel on $\R$, $x\in\R$, $\x\in\bR$, $t\geq0$, the following equations hold,
\begin{align}
	\T_t\wc{f}(\x) &= \wwc{\vrt{(\T_t\wc{f})}{\Rc}}(\x), \label{eq:nice_formula} \\
	\T_t(\wc{f}\wc{g})(\x) &= \wwc{\vrt{(\T_t\wc{f}\wc{g})}{\Rc}}(\x) + \T_t\wc{f}(\x)\, \T_t\wc{g}(\x) -  \wwc{ \vrt{(\T_t\wc{f})}{\Rc} \! \vrt{(\T_t\wc{f})}{\Rc} }(\x). \label{eq:another_nice_formula}
\end{align}
\end{cor}
\begin{proof}
	For $f,g\in\B(\R)$, the statement follows from \eqref{eq:fg}. Next, consider $\min\{f,n\}$ and $\min\{g,n\}$ and apply the monotone convergence theorem, for $n\to\infty$. 
\end{proof}

\begin{proof}[Proof of Theorem \ref{thm:f_fg}]
We assume first that $f,g\in \B(\R)$.
Then, for the function $v$ the statement follows from {\cite[Theorem~4.13]{INW1969iii}}.
Let us prove it for $w$. By Lemma~\ref{lem:prod_of_sum}, for $\gamma\in(0,1]$, 
\[
	w_\gamma(x,t):=\T_t(\gamma\wc{f}\wc{g})(\{x\}),
\]
satisfies the following equation
\begin{equation}\label{eq:w11}
	w_\gamma(x,t) = (T_t^0 \gamma f g)(x) + (\bK \, \wh{\gamma}(\wc{w\!}_\gamma +\wc{v\!}_f \wc{v\!}_g - \wc{v_f v_g}))(x,t).
\end{equation}
	Denote inductively, for $\gamma\in(0,1]$,
	\[
		T_{t}^{\gamma,j} h := \bK\, \wh{\gamma}\,( \wwc{T_t^{\gamma,j-1} h} +\wc{v\!}_f \wc{v\!}_g - \wwc{v_f v_g}),
	\]
	where $T_t^{\gamma,0}h:= T_t^0 (\gamma h)$.
	Then $\sum_{j\in\N_0} T_t^{\gamma,j}\, fg$ is the solution to \eqref{eq:w11}.
	In fact, it is the minimal solution in the class of non-negative functions.
	Indeed, let $w\geq0$ solves \eqref{eq:w11}, then
	$\wc{v\!}_f\wc{v\!}_g-\wwc{v_f v_g} \geq0$ implies  for any $n\in\N_0$, 
	\[
		w(x,t) = \sum_{j=0}^n (T_t^{\gamma,j} fg) + (T_t^{\gamma,n+1} w) \geq \sum_{j=0}^n (T_t^{\gamma,j} fg) .
	\]
	As a result, $w \geq \sum_{j\in\N_0} T_t^{\gamma,j} fg$.
	On the other hand, we proved that for $\gamma\in(0,1)$, $w_\gamma$ is the minimal solution to \eqref{eq:w11}.
  Hence we have,
 	\[
		w_\gamma = \sum_{j\in\N_0} T_t^{\gamma,j} fg, \quad \gamma\in(0,1).	
	\]
	Taking $\gamma\to1_{-}$, we obtain $w = \sum_{j\in\N_0} T_t^{1,j} fg$.
Hence $w$ is the minimal solution to \eqref{eq:w11}.

Let now $f,g$ be (possibly) unbounded. Denote 
\[
	f_m=\min\{f, m\},\quad g_m=\min\{g, m\},\quad w_m(x,t) =\T_t(\wc{f\!}_m\wc{g\!}_m)(\{x\}),\quad  m\in\N.
\]
Since		
\[
	\sum_{j=0}^{n} (T_t^{\gamma,j} f_m g_m) \leq \sum_{j=0}^{n} (T_t^{\gamma,j} f g) \leq \sum_{j=0}^{n} (T_t^{\gamma,j} f g) + (T_t^{\gamma,n+1} w),
\]
then
\[
	w_m(x,t)  = \sum_{j\in\N_0} (T_t^{\gamma,j} f_m g_m)(x,t) \leq \sum_{j\in\N_0} (T_t^{\gamma,j} f g)(x,t) \leq w(x,t). 
\]
On the other hand, by the monotone convergence theorem, for $m\to\infty$,
\[
	w_m(x,t) = \E[\{x\}][\wc{f\!}_m(\X_t)\wc{g\!}_m(\X_t)] \to  \E[\{x\}][\wc{f}(\X_t)\wc{g}(\X_t)] = w(x,t).
\]
As a result,
\[
	\sum_{j\in\N_0} (T_t^{\gamma,j} f g)(x,t) = w(x,t),
\]
and, as before, $w$ is the minimal solution to \eqref{eq:w}. In the same way one can show the $v$ is the minimal solution to \eqref{eq:integral_eq_of_exp_semigroup}.
The proof is fulfilled.
\end{proof}

\section{Some properties of solutions to the S-equation}

Let $\X$ be a $(X^0,\pi)$-branching Markov process given by Definition~\ref{defn:X0_pi_BMP} with the corresponding semi-group $\T$ given by \eqref{eq:MP_semigroup}. 
Assumption~\ref{assumm:doesnt_explode} and Theorem~\ref{thm:BMP_and_PDE} imply the comparison principle for the S-equation: 
\begin{prop}\label{prop:comparison}
	Let Assumption~\ref{assumm:doesnt_explode} hold and $\X$ be a $(X^0,\pi)$-branching Markov process. Suppose 
	\[ 
		0\leq f_1(x) \leq f_2(x) \leq 1,\ x\in\R, \quad   f_1, f_2 \in \B(\R),
	\]
	and $u_1$, $u_2$ be the minimal solutions to the S-equation given by Theorem~\ref{thm:BMP_and_PDE} with the initial conditions $f_1$ and $f_2$ correspondingly. Then 
	\[
		0 \leq u_1(x,t) \leq u_2(x,t) \leq 1,\quad x\in\R,\ t\geq0.
	\]
\end{prop}

For the shift operator on $\R$ we will write $S_y(x) = x+y$, $x,y\in\R$. For $\x\in\bR$, $y\in\R$, with abuse of notations we denote 
	\[
		S_y(\x)=\x+y:=\wh{S_y}(\x),
	\]
where we put $S_y(\varnothing) := \varnothing$.

\begin{defn}\label{defn:spatially_homogeneous}
We call a Markov process spatially homogeneous if its semi-group $\T$ commutes with the shifts of space, namely,
\begin{equation}\label{eq:space_homogeneous}
	\T_t S_y f(\x) = S_y \T_t f(\x), \quad \x\in\bR,\ y\in\R,\ t\geq0,\ f\in\B(\bR).
\end{equation}
\end{defn}
This means in particular that trajectories started from $\X_0=\{y\}$, $y\in\R$, coincide with trajectories started from $\X_0 =\{0\}$, $0\in\R$, and shifted by $y$.

By \eqref{eq:non-branching_part} and item \ref{assum:branching_law} of Definition~\ref{defn:X0_pi_BMP}, if a $(X^0,\pi)$-branching Markov process is spatially homogeneous then its non-branching part $X_0$ and branching law $\pi$ are also spatially homogeneous, namely,
\[
	T_t^0 S_y f(x) = S_y T_t^0 f(x),\qquad \pi(x,E) = \pi(0,-x+E).
\]
\begin{prop}\label{prop:u_is_BUC}
	Let Assumption~\ref{assumm:doesnt_explode} hold and $\X$ be a spatially homogeneous $(X^0,\pi)$-branching Markov process.
Assume that $f \in \B(\R)$ be uniformly continuous and $0\leq f\leq1$.
Then the minimal solution $u(x,t)$ to the S-equation with the initial condition $f$ given by Theorem~\ref{thm:BMP_and_PDE} is a uniformly continuous function jointly in $(x,t)$ on $\R\times\R_+$.
\end{prop}
\begin{proof}
	By Definition~\ref{defn:BMP}, $\X$ is right-continuous, hence so is $X^0$. Thus, since we assumed that $0\leq f \leq 1$ is uniformly continuous, then
\begin{align*}
	\lim_{t\to0_+} \esssup_{x\in\R} |f(x+X^0_t)-f(x)| =0,& \quad \mathbb{P}^0_{0}-a.s.,\\
	\esssup_{x\in\R} |f(x+X^0_t)-f(x)| \leq 1,& \quad \mathbb{P}^0_{0}-a.s.
\end{align*}
Hence, by \eqref{eq:minimal_sol_to_S_equation}, Definition~\ref{defn:X0_pi_BMP}, and the dominated convergence theorem,
\begin{align*} 
	\|\T_t\wh{f}-\wh{f}\| &\leq \E[\{0\}][\|\wh{f}(\cdot + \X_t)-f(\cdot)\|, \tau>t] + \E[\{0\}][\tau\leq t] \\
		&= \mathbb{E}^0_{0}\big[ \|f(\cdot+X^0_t)-f(\cdot)\|] + \E[\{0\}][\tau\leq t] \to 0,\quad t\to0_+; 
\end{align*}
	Then, for $\dx\in\R$, $\dt\in\R_+$,
	\[
		\big|u(x+\dx,t+\dt)-u(x+\dx,t)\big| \leq \|\T_t (\T_{\dt}\wh{f}-\wh{f})\| \leq \|\T_{\dt}\wh{f}-\wh{f}\| \to 0,\quad \dt\to0.  
	\]
	Next, note that for $\x=\{x_1,x_2,\dots,x_n\}$, $\x'=\{x_1,x_2,\dots,x_{n-1}\} $, 
	\begin{align*}
		|\wh{f}(\x+\dx) - \wh{f}(\x)| &= \wh{f}(\x'+\dx) |f(x_n+\dx)-f(x_n)| + f(x_n) |\wh{f}(\x'+\dx)-\wh{f}(\x')|\\
			&\leq \|f\|^{n-1}\|S_{\dx}f-f\| + |\wh{f}(\x'+\dx)-\wh{f}(\x')|.
	\end{align*}
	Therefore, by induction,
	\begin{align*}
		\big| \wh{f}(\x+h)-\wh{f}(\x)\big| \leq (1+\|f\| +\dots \|f\|^{n-1}) \|S_{\dx}f-f\| \leq n \|S_{\dx}f-f\|.
	\end{align*}
	Hence, c.f. \eqref{eq:explosion_and_first_branching}, \eqref{eq:space_homogeneous},
	\begin{align*}
		u(x+\dx,t)-&u(x,t) = \E[\{0\}][S_{\dx}\wh{f}(x+\X_t)-\wh{f}(x+\X_t)]\\
			&\leq \E[\{0\}][S_{\dx}\wh{f}(x+\X_t)-\wh{f}(x+\X_t),\, \tau_n > T] + 2\E[\{0\}][\tau_n \leq T]\\ 
			&\leq \frac{n(n+1)}{2} \|S_{\dx}f-f\|\,\E[\{0\}][\tau_n>T] + 2\E[\{0\}][\tau_n \leq T]. 
	\end{align*}
	Since $f$ is uniformly continuous, then by Lemma~\ref{lem:isnt_extinct_is_necessary}, taking first $\dx\to0$, and next $n\to\infty$, we obtain uniform in $(x,t)$ convergence to $0$. 
	As a result, $u(x+\dx,t+\dt)-u(x,t)$ converges to $0$ uniformly in $(x,t)\in\R\times\R_+$, as $\dx\to0$, $\dt\to0_+$. The proof is fulfilled. 

\end{proof}

\section{On the branching random walk and its relation to the branching Markov process}

To describe the branching random walk on the real line we will follow Shi \cite{Shi2015}. 
We assume that an initial ancestor is located at a point $x\in\R$. 
Its children, who form the first generation, are scattered in $\R$ according to the distribution of the point process $x+\Xi$.
Each of the particles (also called individuals) in the first generation produces
its own children who are thus in the second generation and are positioned (with
respect to their parent) according to the same distribution of $\Xi$.
The system goes on indefinitely, but can possibly die if there is no particle at a generation.
Each individual in the $n$-th generation reproduces independently of
each other and members of earlier generations.

\begin{lem}\label{lem:from_BMP_to_BRW}
  Let $\X$ be a spatially homogeneous $(X^0,\pi)$-branching Markov process.
  Then its sampling $(\X_n):=\{\X_n\}_{n\in\N_0}$ is a branching random walk.	
	Moreover, if $\X$ started from a point $x\in\R$, then the corresponding $(\X_n)$ is a branching random walk on the real line $\R$.  
\end{lem}
\begin{proof}
	Since $\X$ is a Markov process then $(\X_n)$ is a Markov chain.
	The branching property of $(\X_n)$ follows form the branching property of $\X$, namely, by \eqref{eq:branching_property}, 
	\begin{equation}\label{eq:branching_everywhere}
		\T_n\wh{f}(\x) = \wh{(\T_n\wh{f})\big\vert_{\wh{\R}}}(\x).
	\end{equation}
where $f\in\B(\Rc),\ \|f\|<1$.
By \cite[Lemma~0.2]{INW1968i}, in \eqref{eq:branching_everywhere} one can substitute $\wh{f}$ by any continuous function $g:\bR\to\R$, vanishing at infinity (as $\x\to\iD$).
For example, by the spatial homogeneity of $\X$, \eqref{eq:branching_everywhere} implies
\[
	\E[\x][g(\X_1)] = \prod_{x_j\in\x} \E[\{x_j\}][g(\X_1)] = \prod_{x_j\in\x} \E[\{0\}][g(x_j+\X_1)].
\]
This means, that any point $x_j$ in the starting generation produces its own children distributed as $x_j+\X_1$, which are positioned independently of the children of other parents from $\x$.
By induction, it holds for any generation $n\in\N$.
Thus $(\X_n)$ is a branching random walk.
The last statement of the theorem follows from the item \eqref{assum:inftyinfty} of Definition~\ref{defn:X0_pi_BMP}.
The proof is fulfilled.
\end{proof}

\begin{rem}\label{rem:on_GW}
	The idea of Lemma~\ref{lem:from_BMP_to_BRW} is standard in the probabilistic literature. For example, if particles of the $(X^0,\pi)$-branching Markov process $\X$ don't move (see \ref{item:X1}), then the corresponding sampling $(\X_n)$ is the Galton-Watson process. See e.g. \cite[p.\,60]{Sev1951} or \cite[p.\,110]{AN1972}.	
\end{rem}

Let us remind that we denote 
	\[
		\el(y) := e^{- \la y},\quad y\in\R,\ \la\in\R,
	\]
	and the corresponding  expectation semi-group $v_\la$ is defined by \eqref{eq:Laplace_tranform}. 
Then the log-Laplace transform of the point process $\X_1$ is defined by \eqref{eq:log-Laplace}. 

Now we will formulate results on the position of the left-most particle of a branching random walk.

Assume,
\begin{align}
	&\E[\{0\}][\wc{g}_{\!\!1}(\X_1)] = 0, & \quad &\E[\{0\}][\wc{g}_{\!\!2}(\X_1)]<\infty, \label{assum:H1} \tag{H1}\\
	&\E[\{0\}][\wc{g}_{\!\!3}(\X_1) (\ln_+ \wc{g}_{\!\!3}(\X_1))^2]<\infty, & \quad &\E[\{0\}][\wc{g}_{\!\!4}(\X_1)\ln_+ \wc{g}_{\!\!4}(\X_1)]<\infty, \label{assum:H2} \tag{H2}
\end{align}
where, $\ln_+x:= \ln \max\{x,1\}$, and for $y\in\R$, 
\begin{align*}
	&h(y)=\la_* y+\psi(\la_*), & \quad &g_1(y)=h(y) e^{-h(y)}, & \quad &g_2(y)=h^2(y) e^{-h(y)},\\ 
	&g_3(y)=e^{-h(y)}, & \quad &g_4(y)=\max\{0,h(y)\}e^{-h(y)}.
\end{align*}
Let $(M_n)$ denote a position of the left-most particle of $(\X_n)$,
\begin{equation}\label{eq:right-most}
	M_n := \min\{x\in\R\,\vert\, x\in\X_n\},\quad n\in\N_0.
\end{equation}
By \cite[Theorem~1.1]{Aid2013} the following theorem holds.
\begin{thm}\label{thm:lim_of_leftmost_of_BRW}
	Under \eqref{assum:supercritical}, \eqref{assum:c_*}, \eqref{assum:H1}, \eqref{assum:H2}, if the distribution of $\X_1$ is non-lattice, then there exists a constant $C_*>0$, such that for any $x\in\R$,
\[
	\lim\limits_{n\to\infty} \E[\{0\}][M_n + c_* n - \frac{3}{2\la_*}\ln n + C_* \geq x] = \E[\{0\}][e^{-e^{\la_* x}D_\infty}],
\]
where $D_\infty$ is the almost sure limit of the derivative martingale 
\begin{equation}\label{eq:derivative_martingale}
	D_n=\wc{g}(\X_n), \quad g(y) = (\la_* y + n \psi(\la_*))e^{-\la_* y-n \psi(\la_*)}.
\end{equation}
	Moreover,
	\begin{equation}\label{eq:derivative_martingale_is_positive}
		\P[\{0\}][D_\infty>0 \, \big| \, (\X_n)\text{ does not extinct}]=1. 
	\end{equation}
\end{thm}

The assumptions \eqref{assum:H1} and \eqref{assum:H2} are difficult to check in general.
Therefore we introduced the sufficient condition \eqref{assum:sufficient_for_H1_and_H2}.
In order to check that \eqref{assum:sufficient_for_H1_and_H2} implies \eqref{assum:H1} and \eqref{assum:H2} we need the following lemma.
\begin{lem}\label{lem:Laplace_is_time-multiplicative}
	Let $\X$ be a $(X^0,\pi)$-branching Markov process satisfying Assumption~\ref{assumm:doesnt_explode}. 
	\begin{enumerate}
		\item If for fixed $\la\in\R$, $t>0$, the function $v_\la$ given by \eqref{eq:Laplace_tranform} is such that $v_\la(0,t)<\infty$ and $v_0(0,t)<\infty$, then 
	\begin{equation}\label{eq:Laplace_is_time_multiplicative}
		v_{\la}(0,t) = v_{\la}(0,t-s)v_\la(0,s), \quad s\in[0,t].
	\end{equation}
	In particular $v_{\la}(0,s)<\infty$, $s\in[0,t]$.
\item If for fixed $\la,\mu\in\R$, $t>0$, the function $w_{\la,\mu}$ given by \eqref{eq:w_la_mu} is such that $w_{\la,\mu}(0,t)<\infty$, then 
  	\begin{equation}\label{eq:w_leq_v_w}
  		w_{\la,\mu}(0,s) v_{\la+\mu}(0,t-s) \leq w_{\la,\mu}(0,t),  \quad s\in[0,t].
  	\end{equation}
  	In particular $v_{\la+\mu}(0,s)<\infty$, $w_{\la,\mu}(0,s)<\infty$, $s\in[0,t]$.
	\end{enumerate}
\end{lem}
\begin{proof}
	Denote
	\[
		e_{\la,n}(y):= \min\{n, e^{- \la y}\}.
	\] 
	First, we note that $v_{\la}(0,t)<\infty$ implies $v_{\la}(0,s)<\infty$, $s\in[0,t]$.
	Indeed, by the Markov property of $\X$,
	\begin{align*}
		\E[\{0\}][\wc{e_{\la,n^2}}(\X_{t})] &= \E[\{0\}][\E[\{0\}][\wc{e_{\la,n^2}}(\X_{t}) \big\vert\, \X_{t-s}]] = \E[\{0\}][\E[\X_{t-s}][\wc{e_{\la,n^2}}(\X_{t})]],\\ 
	\end{align*}
	by \eqref{eq:nice_formula} (c.f. \eqref{eq:check}) and spatial homogeneity of $\X$ (see Definition~\ref{defn:spatially_homogeneous}) we continue
	\begin{align*}
		&=\E[\{0\}][\wwc{ \vrt{(\T_s\wc{e_{\la,n^2}}\,)}{\R}}(\X_{t-s})]  = \E[\{0\}][\sum_{\xi_j\in\X_{t-s}}\!\! \E[\{0\}][\wc{e_{\lambda,n^2}}(\xi_j+\X_s)]] \\ 
		&\geq \E[\{0\}][\sum_{\xi_j\in\X_{t-s}}\!\! e_{\la,n}(\xi_j)\,\E[\{0\}][\wc{e_{\lambda,n}}(\X_s)]] = \E[\{0\}][\wc{e_{\lambda,n}}( \X_{t-s})]\E[\{0\}][\wc{e_{\lambda,n}}( \X_s)], 
\end{align*}
where we applied the inequality $\min\{n^2, e^{x+y}\} \geq \min\{n, e^x\} \min\{n, e^y\}$, $x,y\in\R$. 
	Taking $n\to\infty$, we have 
	\begin{equation}\label{eq:v2_t_leq_v_2t}
		v_{\la}(0,t-s)v_{\la}(0,s) \leq v_{\la}(0,t) < \infty.
	\end{equation}
	In particular $v_0(0,t)<\infty$ implies $v_0(0,s)<\infty$, $s\in[0,t]$.

	Let us prove the opposite inequality.
	Denote
	\[
		e_{\la,n,m}(y):= \max\{\min\{n, \el(y)\}, \frac{1}{m}\}.
	\] 
	Then, for $m\leq n$, $n,m\in\N$, the inequality holds,
	\[
		e_{\la,n,m^2}(x+y) \leq e_{\la,n^2,m}(x) e_{\la,n^2,m}(y).
	\]
	Hence, similar to the previous case, we have for $n\geq m$,
	\begin{align*}
		\E[\{0\}][\wc{e_{\la,n,m^2}}(\X_{t})] \leq \E[\{0\}][\wc{e_{\lambda,n^2,m}}( \X_{t-s})]\E[\{0\}][\wc{e_{\lambda,n^2,m}}( \X_{s})]. 
	\end{align*}
	Taking first $n\to\infty$ and then $m\to\infty$ we have
	\begin{equation}\label{eq:v2_t_geq_v_2t}
		v_{\la}(0,t) \leq v_{\la}(0,s)v_{\la}(0,t-s),
	\end{equation}
	where the limit on the right-hand side is finite by the following estimate,
	\[
		\E[\{0\}][\wc{e_{\lambda,n^2,m}}( \X_{s})] \leq \E[\{0\}][\wwc{(e_{\lambda,n^2}+1)}( \X_{s})] \leq v_\la(0,s) + v_0(0,s) <\infty, \ \ s\in[0,t].
	\]  
	By \eqref{eq:v2_t_leq_v_2t} and \eqref{eq:v2_t_geq_v_2t} the equation \eqref{eq:Laplace_is_time_multiplicative} holds.

	Similar to the previous consideration, we have by \eqref{eq:nice_formula} and \eqref{eq:another_nice_formula},
	\begin{align*}
		&\E[\{0\}][\wc{e_{\la,n^2}}(\X_t) \wc{e_{\mu,n^2}}(\X_t)] = \E[\{0\}][(\T_s \wc{e_{\la,n^2}} \wc{e_{\mu,n^2}})(\X_{t-s})] \\ 
		& \geq \E[\{0\}][ \big(\T_s(\wc{e_{\la,n^2}}\wc{e_{\mu,n^2}})\big\vert_{\R}\big) \!\wc{\phantom{)}}(\X_{t-s})] = \E[\{0\}][\sum_{\xi_j\in \X_{t-s}} \E[\{\xi_j\}][\wc{e_{\lambda,n^2}}(\X_s)\wc{e_{\mu,n^2}}(\X_s)]]\\ 
		&= \E[\{0\}][\sum_{\xi_j\in \X_{t-s}} \E[\{0\}][\wc{e_{\lambda,n^2}}(\xi_j+\X_s)\wc{e_{\mu,n^2}}(\xi_j+\X_s)]]\\ 
		&\geq \E[\{0\}][\sum_{\xi_j\in \X_{t-s}} e_{\la,n}(\xi_j)e_{\mu,n}(\xi_j)  \E[\{0\}][\wc{e_{\lambda,n}}(\X_s)\wc{e_{\mu,n}}(\X_s)]]\\
		& =\E[\{0\}][\wc{e_{\la,n}e_{\mu,n}}(\X_{t-s})] \E[\{0\}][\wc{e_{\lambda,n}}(\X_s)\wc{e_{\mu,n}}(\X_s)].
	\end{align*}
	Taking $n\to\infty$, we have
	\[
		w_{\la,\mu}(0,t) \geq v_{\la+\mu}(0,t-s) w_{\la,\mu}(0,s).
	\]
The proof is fulfilled.
\end{proof}

\begin{lem}\label{lem:sufficient_for_H1_and_H2}
	The assumption \eqref{assum:sufficient_for_H1_and_H2} implies \eqref{assum:H1} and \eqref{assum:H2}.
\end{lem}
\begin{proof}
	For simplicity we denote (possibly) different constants by the same letter $C>0$.
	By Lemma~\ref{lem:Laplace_is_time-multiplicative}, \eqref{assum:sufficient_for_H1_and_H2} implies
	\[
		v_{0}(0,s) + v_{\delta+\la_*}(0,s) + w_{0,0}(0,s) + w_{0,\la_*}(0,s) + w_{\delta,\la_*}(0,s)<\infty,\qquad s\in[0,1].
	\]
	There exists $C>0$ such that
	\[
		g_2(y) \leq C(1 + e^{-(\la_*+\delta)y}).
	\]
	Therefore,
	\[
		\E[\{0\}][\wc{g}_{\!\!2}(\X_1)] \leq C (v_0(0,1) +  v_{\delta+\la_*}(0,1)) <\infty. 
	\]
	Next, $c_* = \frac{\partial}{\partial \la} \psi(\la_*)$ implies $\frac{\partial}{\partial_\la} v_{\la_*}(0,1) = c_* v_{\la_*}(0,1)$ which yields $\E[\{0\}][\wc{g}_{\!\!1}(\X_1)] = 0$.
	Hence \eqref{assum:H1} holds.
	
	Since, for $x\geq e$, $\partial_x \ln_+^2(x)$ is a non-negative decreasing function, then for $x\geq e$, $y\geq0$,
	\[
		\ln_+^2(x+y) - \ln_+^2(x) = \int[x][x+y] \partial_x \ln_+^2(z)dz \leq \int[e][y+e] \partial_x \ln_+^2(e) = \ln_+^2(y+e) - \ln_+^2(e). 
	\]
	Hence, for $\x\in\bR\backslash\{\varnothing\}$, 
	\[
		\ln_+^2 \wc{g\!}_3(\x) \leq \ln^2 \sum_{y\in\x} \max\{e,g_3(y)\} \leq \sum_{y\in\x} \ln^2 \max\{e,g_3(y)\} + \ln^2(e\cdot \wc{1\!}(\x)).
	\]
	There exists $C>0$ such that $\ln^2(e\cdot x) \leq \frac{C}{2}x$, $x\geq1$, and
	\[
		\ln^2\max\{e,g_3(y)\} \leq \frac{C}{2} (1+e^{-\delta y}),\quad y\in\R.
	\]
	Then, for $\x\in\bR\backslash\{\varnothing\}$, 
	\begin{align*}
		\ln_+^2 \wc{g\!}_3(\x)  \leq C(\,\wc{1\!}(\x) + \wc{e}_{\!\delta}(\x)),
	\end{align*}
	and, finally,
	\[
		\E[\{0\}][\wc{g}_{\!\!3}(\X_1)\ln_+^2 \wc{g}_{\!\!3}(\X_1)]\leq C (w_{0,\la_*}(0,1)+w_{\delta,\la_*}(0,1)) <\infty.
	\]
	By the definition of $g_4$ there exists $C>0$ such that $g_4(y)\leq C$, $y\in\R$. 
	Hence,
  \[
		\E[\{0\}][\wc{g}_{\!\!4}(\X_1)\ln_+ \wc{g}_{\!\!4}(\X_1)] \leq C w_{0,0}(0,1)<\infty.
  \]
	As a result, \eqref{assum:H2} holds and the proof is fulfilled.
\end{proof}

\section{Proof of Theorem~\ref{thm:main}}
Let $u$ be the minimal non-negative solution to the S-equation \eqref{eq:S} with the initial condition  $u(x,0)=f(x) = \1_{\R_+}(x)$.
Such solution is given by Theorem~\ref{thm:BMP_and_PDE} and it satisfies \eqref{eq:minimal_sol_to_S_equation}.
Therefore,
\begin{align*}
	u(x,t) &= \E[\{0\}][\wh{f}(x+\X_t)] = \E[\{0\}][\forall y\in\X_t:\ x+y\geq 0] \\  
	&= \E[\{0\}][M_t:=\min\{y\in\R:\,y\in\X_t\} \geq -x] ,\quad x\in\R,\ t\geq0,
\end{align*}
where we remind the reader that non of the points in $\X$ equals $\infty$ by the item \ref{assum:inftyinfty} in Definition~\ref{defn:X0_pi_BMP}.
Hence, \eqref{eq:u_and_M} holds true. Then, \eqref{eq:left-most_limiting_law} is equivalent to \eqref{eq:stability}. Therefore, it is sufficient to prove the latter one.

Let us consider for $k\in\N$ the branching random walk on $\R$, c.f. Lemma~\ref{lem:from_BMP_to_BRW}, 
\[
	(\X_n(k)):= \{ \X_{\frac{n}{2^k}} \}_{n\in\N_0}.
\]
By Lemma~\ref{lem:Laplace_is_time-multiplicative}, the log-Laplace transform $\psi_k$ of $\X_1(k)$ satisfies
\[
	2^k \psi_k(\la) = \psi_1(\la) = \psi(\la).
\]
Therefore, \eqref{assum:supercritical} and \eqref{assum:c_*} for $\psi$ imply analogous assumptions for $\psi_k$, with the same $\la$ and $\la_*$, namely,
\[
 \psi_k(0)\in(0,\infty),\quad \psi_k(\la)<\infty,\quad \frac{\psi_k(\la_*)}{\la_*} = \inf_{\la>0} \frac{\psi_k(\la)}{\la} = \frac{c_*}{2^k}.
\]
The corresponding derivative martingale $(D_n(k))$ satisfies 
\[
	D_n(k)\to D_\infty(k),\ n\to\infty,\ a.s.
\]
Since $D_{2^kn}(k) = D_n(1)$, then a.s. $D_\infty(k)=D_\infty(1)=D_\infty$, $k\in\N$.
Denote
\[
	M_n(k) := \min\{x\in\R\,\vert\, x\in\X_n(k)\},\quad n\in\N_0.
\]
Then by Theorem~\ref{thm:lim_of_leftmost_of_BRW}, for any $k\in\N$ there exists $C_k$ such that 
\[
	\lim\limits_{n\to\infty} \E[\{0\}][M_n(k) + \frac{c_* }{2^k}n - \frac{3}{2\la_*}\ln n + C_k \geq x] = \E[\{0\}][e^{-e^{\la_* x}D_\infty}].
\]
Since for $n=m2^k$, $m,k\in\N$, the left hand side coincides with the one in Theorem~\ref{thm:lim_of_leftmost_of_BRW}, then $C_k = C_* + \frac{3}{2}k\ln2$. 
As a result, by Theorem~\ref{thm:BMP_and_PDE},
\begin{equation}\label{eq:descrete_k_limit}
	\lim_{n\to\infty} u(x + \theta\big(\frac{n}{2^k}\big), \frac{n}{2^k}) = \E[\{0\}][e^{-e^{-\la_* x}D_\infty}] := \phi(x), \quad x\in\R,\ k\in\N,
\end{equation}
where $u(x,t)$ is the minimal solution to the S-equation \eqref{eq:S} with the initial condition $u_0(x)=\1_{\R_+}(x)$, and 
\[
	\theta(t) := c_*t - \frac{3}{2\la_*}\ln t + C_*.
\]
Consider a uniformly continuous function  $f_h$, such that
\[
	f(x)  \leq f_h(x) \leq f(x+h), \quad x\in\R.
\]
Denote $u_h(x,t)=\vrt{\big(\T_t \wh{f_h}\big)}{\R}\!(x)$.
By Theorem~\ref{prop:comparison}, 
\[
	u(x, t) \leq u_h(x, t) \leq u(x + h, t), \quad x\in\R,\ t\geq0.
\]
By Proposition~\ref{prop:u_is_BUC}, $u_h$ is uniformly continuous in $(x,t)\in\R\times\R_+$. Hence, taking $k\to\infty$ in \eqref{eq:descrete_k_limit},
\[
	\phi(x) \leq \liminf_{t\to\infty} u_h(x+\theta(t),t) \leq \limsup_{t\to\infty} u_h(x+\theta(t),t) \leq \phi(x+h).
\]
On the other hand, $u_h(x-h, t) \leq u(x, t) \leq u_h(x, t)$, implies
\[
	\phi(x-h) \leq \liminf_{t\to\infty} u(x+\theta(t),t) \leq \limsup_{t\to\infty} u(x+\theta(t),t) \leq \phi(x+h).
\]
Since $\phi$ is continuous, then taking $h\to0_+$ we obtain \eqref{eq:stability}.

Let us now show that $\phi(x) =  \E[\{0\}][e^{-e^{-\la_* x}D_\infty}]$ is a monotone traveling wave profile with speed $c_*$ to the S-equation, namely that $(x,t)\to\phi(x-c_*t)$ satisfies \eqref{eq:S} and $\phi(+\infty)=1$, $\phi(-\infty)=\P[\{0\}][D_\infty=0]<1$. 
The letter two limits hold by \eqref{eq:traveling_wave}. 
Since \eqref{assum:supercritical} implies (see e.g. \cite{Shi2015}), 
\[
	\P[\{0\}][(X_n)\text{ does not extinct}]>0,
\]
then by \eqref{eq:derivative_martingale_is_positive} and \eqref{assum:supercritical}, $\P[\{0\}][D_\infty=0]<1$.

Monotonicity of $\phi$ is obvious. 
By \eqref{eq:stability}, \eqref{eq:minimal_sol_to_S_equation} and Definition~\ref{defn:spatially_homogeneous}, for $u(x,0)=f(x)=\1_{\R_+}(x)$, $t_0\geq0$,
\begin{align*}
	\phi(x) &= \lim_{t\to\infty} u(x+c_*(t+t_0) - \frac{3}{2\la_*} \ln (t+t_0) + C,t+t_0)\\ 
	&= \lim_{t\to\infty} \vrt{(\T_{t_0} S_{c_*t_0} \T_{t} \wh{f}\,)}{\R} (x+c_*t- \frac{3}{2\la_*} \ln t - \frac{3}{2\la_*} \ln(1+\frac{t_0}{t})+ C) \\
	&= \lim_{t\to\infty} \vrt{(\T_{t_0} S_{c_*t_0} \T_{t} \wh{f}\,)}{\R} (x+c_*t- \frac{3}{2\la_*} \ln t + C)  = \vrt{(\T_{t_0} S_{c_*t_0} \wh{\phi})}{\R}(x). 
\end{align*}

This finishes the proof of Theorem~\ref{thm:main}.

\appendix
\section{Proof of Proposition~\ref{prop:nl+logistic}}
\begin{proof}
	By construction, the process $\X$ is spatially homogeneous, Assumption~\ref{assumm:doesnt_explode} holds true, and $\X_1$ is non-lattice. 
	The assumption \eqref{assum:supercritical} holds if and only if there exists $\la>0$, such that $(\L a)(\la)<\infty$.
	By \eqref{eq:safficient_for_pure_jump_logistic}, 
	\[
		\frac{(\L a)(\la)}{\la} \geq \frac{\delta I e^{\la l}}{l} \to \infty,\ \la\to\infty,\quad \frac{(\L a)(\la)}{\la} \sim \frac{1}{\la} \to\infty,\ \la\to 0_+. 
	\]
	Note that $\la_0>0$ defined by \eqref{eq:la_0} is the abscissa of the Laplace transform $\L a$ of $a$.
The function $\frac{(\L a)(\la)}{\la}$ is convex (and hence continuous) on $(0,\la_0)$, since 
	\[
		\frac{\partial^2}{\partial \la^2} \Big(\frac{\psi(\la)}{\la} \Big) = \frac{\int[\R] ((\la x+1)^2+1) e^{-\la x} a(x) dx}{\la^3}>0.
	\]
	Therefore, if the infimum in \eqref{assum:c_*} is not reached at $\la_*=\la_0$, then \eqref{assum:c_*} holds true.
	Next, for $\la\in(0,\la_0)$,
	\[	
		\frac{\partial}{\partial \la} \Big(\frac{\psi(\la)}{\la} \Big) = - \frac{\int[\R](\la x+1)e^{-\la x} a(x)dx}{\la^2} =0 \quad \Leftrightarrow \quad \frac{\partial \psi}{\partial \la}(\la) = \frac{\psi(\la)}{\la}. 
	\]
	Therefore, $c_* =\frac{\partial}{\partial \la} \psi(\la_*)$.
		Next, for $\la,\mu\geq0$, $\la+\mu<\la_0$, 
	\[
		w_{\la,\mu}(x,t) = e^{-(\la+\mu)x} \Big( e^{t(\L a)(\la+\mu)} + \int[0][t] e^{(t-s)(\L a)(\la+\mu)} e^{s(\L a)(\la)} e^{s(\L a)(\mu)}\Big)<\infty,
	\]
	and \eqref{assum:sufficient_for_H1_and_H2} is satisfied. 
	The proof is fulfilled.
\end{proof}

\end{document}